\documentclass[11pt,sfbold,
]{paper}
\usepackage{amsfonts}
\usepackage{amsmath}
\usepackage{theorem}
\usepackage{authordate1-4}
\usepackage{fancyhdr}
\usepackage{geometry}
\usepackage{setspace}

\newtheorem{theorem}{Theorem}
\newtheorem{example}[theorem]{Example}

\newtheorem{corollary}[theorem]{Corollary}

\newtheorem{definition}[theorem]{Definition}

\newtheorem{lemma}[theorem]{Lemma}

\newtheorem{remark}[theorem]{Remark}

\let\tsum\sum

\begin{document}

\author{Edgar Delgado-Eckert\thanks{%
Centre for Mathematical Sciences, Munich University of Technology,
Boltzmannstr.3, 85747 Garching, Germany.} \thanks{%
Pathology Department, Tufts University, 150 Harrison Av., Boston, MA 02111,
USA (correspondence address).} \thanks{%
The author acknowledges support by a Public Health Service grant (RO1
AI062989) to David Thorley-Lawson at Tufts University, Boston, MA.}}
\title{Canonical representatives for residue classes of a polynomial ideal
and orthogonality}
\maketitle

\begin{abstract}
The aim of this paper is to unveil an unexpected relationship between the
normal form of a polynomial with respect to a polynomial ideal and the more
geometric concept of orthogonality. We present a new way to calculate the
normal form of a polynomial with respect to a polynomial ideal I in the ring
of multivariate polynomials over a field K, provided the field K is finite
and the ideal I is a vanishing ideal. In order to use the concept of
orthogonality, we introduce a symmetric bilinear form on a vector space over
a finite field.
\end{abstract}

%\begin{titlepage}

\begin{keywords}
   Polynomial algebras, polynomial ideals, Gr\"obner bases, inner products, normal form 
 \end{keywords}

\begin{PACS}
   13P10; 15A63
 \end{PACS}

%\end{titlepage}

\section{\protect\bigskip Introduction}

A well known result of B. Buchberger is the existence of the normal form of
a polynomial with respect to a polynomial ideal $I$ in the ring of
multivariate polynomials over a field $K$. This result follows from the
existence of so called Gr\"{o}bner bases for polynomial ideals. For a given
fixed term ordering, this normal form is unique \cite{806353}, \cite%
{MR0463136}, \cite{MR0268178}. In this paper we present a new way to
calculate this normal form, provided the field $K$ is finite and the ideal $%
I $ is a vanishing ideal, i.e. $I$ is equal to the set of polynomials which
vanish in a given set of points $X$. Our method doesn't pursue establishing
a new, especially efficient, algorithm for the computation of such a normal
form. Rather, the aim of this paper is to unveil an interesting way to look
at this issue based on the concept of orthogonality.

For orthogonality to apply, we introduce a symmetric bilinear form on a
vector space (see, for instance, \cite{SCHARLAU}). A symmetric bilinear form
can be seen as a generalized inner product. Some authors have explored
vector spaces endowed with generalized forms of inner products. For example,
we refer to the following papers: \cite{MR0133024}, \cite{MR0385527}%
,\linebreak \cite{MR0377485}, \cite{MR0482441}, \cite{MR578592}, \cite%
{MR586525}, \cite{MR2064794}.

Having defined a symmetric bilinear form, we are able to introduce the
notion of orthogonality and orthonormality. Then we consider the orthogonal
solution of a solvable inhomogeneous under-determined linear operator
equation. If one thinks of an inhomogeneous under-determined system of
linear equations in an Euclidean space, the orthogonal solution is simply
the solution that is perpendicular to the affine subspace associated with
the system. After going through existence and uniqueness considerations, we
come to the main statement of this paper, namely, that the above mentioned
normal form can be obtained as the orthogonal solution of a system of linear
equations. That system of equations arises as a linear formulation of the
multivariate polynomial interpolation problem.

Based on our literature research, we believe that the study of polynomial
algebras in the framework of symmetric bilinear spaces (vector spaces
endowed with a symmetric bilinear form) represents a novel approach.
Suitable extensions of our method to more general fields (i.e. infinite
fields) could open new possibilities for studying problems in the areas of
polynomial algebra, computational algebra and algebraic geometry using
functional analytic or linear algebraic techniques.

The concept of orthogonal solution is not limited by monomial orders, as it
is the case for Gr\"{o}bner bases calculations. In this sense, our method
reveals a wider class of normal forms (with respect to vanishing ideals) in
which the normal forms \`{a} la Buchberger appear as special cases.

Another application that we will describe in detail elsewhere is the problem
of choosing a particular interpolant among all possible solutions of a
highly under-determined multivariate interpolation problem. This is related
to the study of the performance of so called "reverse engineering"
algorithms such as the one presented in \cite{MR2086931}.

The organization of this article is the following:

Section 2 is devoted to the general definition of \emph{symmetric bilinear
spaces} and \emph{orthogonal solutions} of an inhomogeneous linear operator
equation. Subsection 2.1 covers basic definitions and properties of
symmetric bilinear spaces, in particular, the concepts of \emph{orthogonality%
} and \emph{orthonormality} are introduced. Subsection 2.2 introduces the
notion of orthogonal solution of a solvable under-determined linear operator
equation. Existence and uniqueness of orthogonal solutions are proved and
some issues regarding the existence of orthonormal bases are discussed.

Section 3 deals with the vector space of functions $F:K^{n}\rightarrow K,$
where $K$ is a finite field and $n\in 
%TCIMACRO{\U{2115} }%
%BeginExpansion
\mathbb{N}
%EndExpansion
.$ In subsection 3.1 we paraphrase the known result that all the functions
in that space are polynomial functions. Subsection 3.2 introduces a linear
operator called \emph{evaluation epimorphism} and formulates the
multivariate polynomial interpolation problem in a linear algebraic fashion.

Section 4 covers the more technical aspect of constructing special symmetric
bilinear forms. Using that type of symmetric bilinear form will allow us to
prove the main result of this article in section 5.

Section 5 is devoted to the statement and proof of our main result. Namely,
that the canonical normal form of an arbitrary polynomial $f$ with respect
to a vanishing ideal $I(X)$ in the ring of multivariate polynomials over a
finite field $K$ can be calculated as the orthogonal solution of a linear
operator equation involving the evaluation epimorphism.

For standard terminology, notation and well known results in computational
algebraic geometry and commutative algebra we refer to \cite{MR1417938} and 
\cite{MR1213453}.

\section{Symmetric bilinear vector spaces and orthogonal solutions of
inhomogeneous systems of linear equations}

\subsection{Basic definitions}

In this subsection we will introduce the concept of a symmetric bilinear
form in a vector space. With this concept it will be possible to define
symmetric bilinear vector spaces and orthonormality. Furthermore, some basic
properties are briefly reviewed (cf. \cite{SCHARLAU})

\begin{definition}
Let $V$ be a vector space over a field $K.$ A\emph{\ }symmetric and bilinear
mapping%
\begin{equation*}
\left\langle \cdot ,\cdot \right\rangle :V\times V\rightarrow K
\end{equation*}%
is called \emph{symmetric bilinear form }on $V.$
\end{definition}

\begin{definition}[Notational Definition]
Let be $n\,,m\in 
%TCIMACRO{\U{2115} }%
%BeginExpansion
\mathbb{N}
%EndExpansion
$ natural numbers and $K$ a field. The set of all $m\times n$ matrices ($m$%
\emph{\ rows and }$n$\emph{\ columns}) with entries in $K$ is denoted by $%
M(m\times n;$ $K).$
\end{definition}

\begin{remark}
\label{MatrixDefOfBilinForms}Let $V$ be a finite dimensional vector space
over a field $K$. After fixing a basis $(u_{1},...,u_{d})$ of $V,$ it is a
well known result, that there is a one-to-one correspondence between the set
of all symmetric bilinear forms on $V$ and the set of all $d\times d$
symmetric matrices with entries in $K$ seen as representing matrices with
respect to the basis $(u_{1},...,u_{d}).$
\end{remark}

\begin{definition}
A vector space $V$ over a field $K$ endowed with a symmetric bilinear form%
\begin{equation*}
\left\langle \cdot ,\cdot \right\rangle :V\times V\rightarrow K
\end{equation*}%
is called a \emph{symmetric bilinear space}.
\end{definition}

\begin{example}
Every (real) Euclidean space is due to the positive definiteness of its
inner product a symmetric bilinear space.
\end{example}

Given a symmetric bilinear space $V$ over a field $K$, \emph{orthogonality}
and \emph{orthonormality} of two vectors $v,w\in V$ as well as the concept
of \emph{orthonormal basis} are defined exactly as in the Euclidean case.
Similarly, the \emph{orthogonal complement} $W^{\perp }:=\left\{ v\in V\mid
v\perp w\text{ }\forall \text{ }w\in W\right\} $ of a subspace $W\subseteq V$
\ is a subspace of $V.$ Furthermore, if $\left( w_{1},...w_{d}\right) $ is
an orthonormal basis of $V,$ then for every vector $v\in V$ holds%
\begin{equation*}
v=\tsum_{k=1}^{d}\left\langle v,w_{k}\right\rangle w_{k}
\end{equation*}%
where the field elements $\left\langle v,w_{i}\right\rangle \in K,$ $%
i=1,...,d$ are the well known \emph{Fourier coefficients}. Contrary to the
case of Euclidean or unitary vector spaces, in symmetric bilinear spaces
orthonormal bases don't always exist.

\begin{example}
\label{Const.Of.Orthonorm.Basis}Let $d\in 
%TCIMACRO{\U{2115} }%
%BeginExpansion
\mathbb{N}
%EndExpansion
$ be a natural number and $V$ a $d$-dimensional vector space over a field $%
K. $ Furthermore let $\left( u_{1},...u_{d}\right) $ be a basis of $V.$ Then
one can construct a symmetric bilinear form on $V$ by setting%
\begin{equation*}
\left\langle u_{i},u_{j}\right\rangle :=\delta _{ij}\text{ }\forall \text{ }%
i,j\in \{1,...,d\}
\end{equation*}%
(see also Remark \ref{MatrixDefOfBilinForms}.) Here the basis $\left(
u_{1},...u_{d}\right) $ is obviously orthonormal.
\end{example}

\subsection{Orthogonal solutions of inhomogeneous linear operator equations}

\begin{definition}
Let $d\in 
%TCIMACRO{\U{2115} }%
%BeginExpansion
\mathbb{N}
%EndExpansion
$ be a natural number and $V$ a $d$-dimensional symmetric bilinear space
over a field $K.$ Furthermore, let $W$ be an arbitrary vector space over the
field $K$, $T:V\rightarrow W$ a non-injective linear operator and $w\in W$ a
vector with the property%
\begin{equation*}
w\in T(V)
\end{equation*}%
Now let $m:=$nullity$(T)\in 
%TCIMACRO{\U{2115} }%
%BeginExpansion
\mathbb{N}
%EndExpansion
$ be the dimension of the kernel of $T.$ A solution $v^{\ast }\in V$ of the
equation%
\begin{equation*}
Tv=w
\end{equation*}%
is called \emph{orthogonal solution}, if for an arbitrary basis $%
(u_{1},...,u_{m})$ of $\ker (T)$ the following orthogonality conditions hold%
\begin{equation*}
\left\langle u_{i},v^{\ast }\right\rangle =0\text{ }\forall \text{ }i\in
\{1,...,m\}
\end{equation*}
\end{definition}

\begin{remark}
\label{Orth.Sol.LiesOnOrth.Comp.}Let $(u_{1},...,u_{m})$ be a basis of $\ker
(T).$ Then each arbitrary vector $u\in \ker (T)$ can be written in the form%
\begin{equation*}
u=\tsum_{i=1}^{m}\lambda _{i}u_{i}
\end{equation*}%
with suitable field elements $\lambda _{i}\in K.$ If the orthogonality
conditions%
\begin{equation*}
\left\langle u_{i},v^{\ast }\right\rangle =0\text{ }\forall \text{ }i\in
\{1,...,m\}
\end{equation*}%
hold for the basis $(u_{1},...,u_{m}),$ then we have%
\begin{equation*}
\left\langle u,v^{\ast }\right\rangle =\left\langle \tsum_{i=1}^{m}\lambda
_{i}u_{i},v^{\ast }\right\rangle =\tsum_{i=1}^{m}\lambda _{i}\left\langle
u_{i},v^{\ast }\right\rangle =0
\end{equation*}%
and that means%
\begin{equation*}
v^{\ast }\in \ker (T)^{\perp }
\end{equation*}%
In particular, for any other different basis $(w_{1},...,w_{m})$ of $\ker
(T) $ it holds%
\begin{equation*}
\left\langle w_{j},v^{\ast }\right\rangle =0\text{ }\forall \text{ }j\in
\{1,...,m\}
\end{equation*}
\end{remark}

\begin{theorem}
Let $d\in 
%TCIMACRO{\U{2115} }%
%BeginExpansion
\mathbb{N}
%EndExpansion
$ be a natural number and $V$ a $d$-dimensional symmetric bilinear space
over a field $K.$ Furthermore, let $W$ be an arbitrary vector space over the
field $K$, $T:V\rightarrow W$ a non-injective linear operator and $w\in W$ a
vector with the property%
\begin{equation*}
w\in T(V)
\end{equation*}%
If $\ker (T)$ has an \emph{orthonormal basis}, then the equation%
\begin{equation*}
Tv=w
\end{equation*}%
has always a unique orthogonal solution $v^{\ast }\in V.$
\end{theorem}

\begin{proof} Let $m:=$nullity$(T)=\dim (\ker (T))\in 
%TCIMACRO{\U{2115} }%
%BeginExpansion
\mathbb{N}
%EndExpansion
$ be the dimension of the null space of $T$ and $(u_{1},...,u_{m})$ an
orthonormal basis of $\ker (T).$ Since $w\in T(V),$ there must exist a
solution $\widehat{\xi }\in V$ of $Tv=w.$ For any other solution $\xi \in V$
we have%
\begin{equation*}
T(\xi -\widehat{\xi })=T(\xi )-T(\widehat{\xi })=0
\end{equation*}%
and therefore%
\begin{equation*}
\xi -\widehat{\xi }\in \ker (T)
\end{equation*}%
That means that all solutions $\xi \in V$ of $Tv=w$ can be written in the
form%
\begin{equation*}
\xi =\widehat{\xi }+\tsum_{i=1}^{m}\lambda _{i}u_{i}
\end{equation*}%
with the $\lambda _{i}\in K,$ $i=1,...,m$ running over all $K.$ In
particular, we can construct a very specific solution by choosing the
parameters $\lambda _{i}\in K,$ $i=1,...,m$ in the following manner%
\begin{equation*}
\lambda _{i}:=-\left\langle u_{i},\widehat{\xi }\right\rangle ,\text{ }%
i=1,...,m
\end{equation*}%
For this solution%
\begin{equation*}
v^{\ast }:=\widehat{\xi }+\tsum_{i=1}^{m}-\left\langle u_{i},\widehat{\xi }%
\right\rangle u_{i}
\end{equation*}%
and for every $j\in \{1,...,m\}$ it holds%
\begin{eqnarray*}
\left\langle u_{j},v^{\ast }\right\rangle &=&\left\langle u_{j},\widehat{\xi 
}+\tsum_{i=1}^{m}-\left\langle u_{i},\widehat{\xi }\right\rangle
u_{i}\right\rangle =\left\langle u_{j},\widehat{\xi }\right\rangle
+\tsum_{i=1}^{m}-\left\langle u_{i},\widehat{\xi }\right\rangle \left\langle
u_{j},u_{i}\right\rangle \\
&=&\left\langle u_{j},\widehat{\xi }\right\rangle
+\tsum_{i=1}^{m}-\left\langle u_{i},\widehat{\xi }\right\rangle \delta
_{ji}=\left\langle u_{j},\widehat{\xi }\right\rangle -\left\langle u_{j},%
\widehat{\xi }\right\rangle =0
\end{eqnarray*}%
This shows the existence of an orthogonal solution of $Tv=w.$ Now let $%
\widetilde{v}\in V$ be another orthogonal solution of $Tv=w.$ Again, since%
\begin{equation*}
T(v^{\ast }-\widetilde{v})=T(v^{\ast })-T(\widetilde{v})=0
\end{equation*}%
we can write%
\begin{equation*}
v^{\ast }=\widetilde{v}+\tsum_{i=1}^{m}\alpha _{i}u_{i}
\end{equation*}%
with suitable $\alpha _{i}\in K.$ From the orthogonality conditions for $%
v^{\ast }$ and $\widetilde{v}$ we have $\forall $ $j\in \{1,...,m\}$%
\begin{eqnarray*}
0 &=&\left\langle u_{i},v^{\ast }\right\rangle =\left\langle u_{i},%
\widetilde{v}+\tsum_{i=1}^{m}\alpha _{i}u_{i}\right\rangle =\left\langle
u_{j},\widetilde{v}\right\rangle +\left\langle u_{j},\tsum_{i=1}^{m}\alpha
_{i}u_{i}\right\rangle \\
&=&\tsum_{i=1}^{m}\alpha _{i}\left\langle u_{j},u_{i}\right\rangle
=\tsum_{i=1}^{m}\alpha _{i}\delta _{ji}=\alpha _{j}
\end{eqnarray*}%
and that means $v^{\ast }=\widetilde{v}.$\end{proof}

\begin{remark}
The existence of an orthonormal basis of $\ker (T)$ is crucial for the proof
of this theorem. It is important to notice that in a symmetric bilinear
space over a general field $K,$ the Gram-Schmidt orthonormalization only
works if the norm%
\begin{equation*}
\left\Vert v\right\Vert :=\sqrt{\left\langle v,v\right\rangle }
\end{equation*}%
of the vectors used in the Gram-Schmidt process exists in the field $K\ $and
is not equal to the zero element. In general terms, the existence of square
roots would be assured in a field $K$ which satisfies%
\begin{equation}
\forall \text{ }x\in K\text{ }\exists \text{ }y\in K\text{ such that }y^{2}=x
\label{ExistenceOfSquareRoots}
\end{equation}%
Now, if $K$ is finite, then (\ref{ExistenceOfSquareRoots}) holds if and only
if $Char(K)=2.$\newline
After fixing a basis $(u_{1},...,u_{d})$ for the vector space $V,$ the
question whether $\left\langle v,v\right\rangle =0$ for $v\neq 0$ is
equivalent to the nontrivial solvability in $K^{d}$ of the following
quadratic form%
\begin{equation}
\vec{x}^{t}A\vec{x}=0  \label{QuadraticForm}
\end{equation}%
where $A$ is the representing matrix of $\left\langle .,.\right\rangle $\
with respect to the basis \emph{\ }$(u_{1},...,u_{d})$ (see Remark \ref%
{MatrixDefOfBilinForms}). In chapter 3, \S 2 of \cite{MR1429394} explicit
formulas for the exact number of solutions in $K^{n}$ of equations of the
type (\ref{QuadraticForm}), where $A$ is a $n\times n$ symmetric matrix with
entries in a finite field $K$, can be found.
\end{remark}

\begin{corollary}
\label{ZeroSol}Let $K,$ $d,$ $V,$ $W$ and $T$ be as in the theorem above. If 
$\ker (T)$ has an \emph{orthonormal basis}, then the equation%
\begin{equation*}
Tv=0
\end{equation*}%
has always the unique orthogonal solution $0\in V.$
\end{corollary}

\section{The vector space of functions $\mathbf{F}_{q}^{n}\rightarrow 
\mathbf{F}_{q}$}

In the next subsection we review the well known result that any function $%
F:K^{n}\rightarrow K,$ where $K$ is a finite field and $n\in 
%TCIMACRO{\U{2115} }%
%BeginExpansion
\mathbb{N}
%EndExpansion
$, is a polynomial function. Furthermore, we introduce the family of
fundamental monomial functions.

\subsection{The ring of polynomial functions in $n$ variables over $\mathbf{F%
}_{q}$ and the vector space of functions $\mathbf{F}_{q}^{n}\rightarrow 
\mathbf{F}_{q}$}

\begin{definition}
We will denote a finite field with $\mathbf{F}_{q}$, where $q$ stands for
the number of elements of the field ($q$ is a power of the prime
characteristic of the field).
\end{definition}

\begin{definition}[Notational definition]
We call a commutative Ring $(R,+,\cdot )$ with multiplicative identity $%
1\neq 0$ and the binary operations $\cdot $ and $+$ just Ring $R$.
\end{definition}

The following three results are well known:

\begin{theorem}[and Definition]
Let $R$ be a ring and $n\in 
%TCIMACRO{\U{2115} }%
%BeginExpansion
\mathbb{N}
%EndExpansion
$ a natural number. The set 
\begin{equation*}
PF_{n}(R):=\{g\text{ }|\text{ }g:R^{n}\rightarrow R\text{ is polynomial}\}
\end{equation*}%
together with the common operations $+$ and $\cdot $ of addition and
multiplication of mappings is a ring. This ring is called \emph{ring of all
polynomial functions over }$R$\emph{\ in }$n$\emph{\ }$R$\emph{-valued
variables.}
\end{theorem}

\begin{theorem}[and Definition]
Let $K$ be an arbitrary field and $n\in 
%TCIMACRO{\U{2115} }%
%BeginExpansion
\mathbb{N}
%EndExpansion
$ a natural number. The set of all functions%
\begin{equation*}
f:K^{n}\rightarrow K
\end{equation*}%
together with the common operations of addition of mappings and scalar
multiplication is a vector space over $K$. We denote this vector space with $%
F_{n}(K).$
\end{theorem}

\begin{theorem}
\label{ThmFuncIsPolFunc}Let $\mathbf{F}_{q}$ be a finite field. Then for the 
\emph{sets} $F_{n}(\mathbf{F}_{q})$ and $PF_{n}(\mathbf{F}_{q})$ it holds%
\begin{equation*}
F_{n}(\mathbf{F}_{q})=PF_{n}(\mathbf{F}_{q})
\end{equation*}
\end{theorem}

\begin{proof} This result is proved in Chapter 7, Section 5 of \cite{MR1429394}. 
\end{proof}

\begin{definition}
Let $n,q\in 
%TCIMACRO{\U{2115} }%
%BeginExpansion
\mathbb{N}
%EndExpansion
$ be natural numbers. Further let $>$ be a total ordering on $\left( 
%TCIMACRO{\U{2115} }%
%BeginExpansion
\mathbb{N}
%EndExpansion
_{0}\right) ^{n}.$ The according to $>$ decreasingly ordered set 
\begin{equation*}
M_{q}^{n}:=\left\{ \alpha \in \left( 
%TCIMACRO{\U{2115} }%
%BeginExpansion
\mathbb{N}
%EndExpansion
_{0}\right) ^{n}\mid \alpha _{j}<q\text{ }\forall \text{ }j\in
\{1,...,n\}\right\}
\end{equation*}%
of all $n$-tuples with entries smaller than $q$ is denoted by $%
M_{q}^{n}\subset \left( 
%TCIMACRO{\U{2115} }%
%BeginExpansion
\mathbb{N}
%EndExpansion
_{0}\right) ^{n}.$
\end{definition}

\begin{remark}
\label{KardinalitaetVonMp}In order to avoid a too complicated notation, we
skip the appearance of the order relation $>$ in the symbol for this set. It
is easy to prove, that $M_{q}^{n}$ contains exactly $q^{n}$ $n$-tuples. We
will index the $n$-tuples in $M_{q}^{n}$ starting with the biggest and
ending with the smallest:%
\begin{equation*}
\alpha _{1}>\alpha _{2}>...>\alpha _{q^{n}}
\end{equation*}
\end{remark}

\begin{definition}
For any fixed natural numbers $n,q\in 
%TCIMACRO{\U{2115} }%
%BeginExpansion
\mathbb{N}
%EndExpansion
$ and for each multi index $\alpha \in M_{q}^{n}$ consider the monomial
function%
\begin{eqnarray*}
g_{nq\alpha } &:&K^{n}\rightarrow K \\
\vec{x} &\mapsto &g_{nq\alpha }(\overrightarrow{x}):=\overrightarrow{x}%
^{\alpha }
\end{eqnarray*}%
All these monomial functions $g_{nq\alpha },$ $\alpha \in M_{q}^{n}$ are
called \emph{fundamental monomial functions}.
\end{definition}

The following result is elementary. Its easy induction proof is left to the
reader:

\begin{theorem}
\label{Fund.Mon.Fct.AreBasis}A basis for the vector space $F_{n}(\mathbf{F}%
_{q})$ is given by the fundamental monomial functions%
\begin{equation*}
(g_{nq\alpha })_{\alpha \in M_{q}^{n}}
\end{equation*}
\end{theorem}

\begin{remark}
The basis elements in the basis $(g_{nq\alpha })_{\alpha \in M_{q}^{n}}$ are
ordered according to the order relation $>$ used to order the $n$-tuples in
the set $M_{q}^{n}.$ That means (see Remark \ref{KardinalitaetVonMp})%
\begin{equation*}
(g_{nq\alpha })_{\alpha \in M_{q}^{n}}=(g_{nq\alpha _{i}})_{i\in
\{1,...,q^{n}\}}
\end{equation*}
\end{remark}

\subsection{Solving the polynomial interpolation problem in $PF_{n}(\mathbf{F%
}_{q})$}

In this subsection we define the \emph{evaluation epimorphism} of a tuple $(%
\vec{x}_{1},...,\vec{x}_{m})\in (\mathbf{F}_{q}^{n})^{m}$ of points in the
space $\mathbf{F}_{q}^{n}.$ The evaluation epimorphism allows for a linear
algebraic formulation of the multivariate polynomial interpolation problem.

\begin{theorem}[and Definition]
Let $\mathbf{F}_{q}$ be a finite field and $n,m\in 
%TCIMACRO{\U{2115} }%
%BeginExpansion
\mathbb{N}
%EndExpansion
$ natural numbers with\linebreak $m\leq q^{n}$. Further let%
\begin{equation*}
\vec{X}:=(\vec{x}_{1},...,\vec{x}_{m})\in (\mathbf{F}_{q}^{n})^{m}
\end{equation*}%
be a tuple of $m$ \textbf{different} $n$-tuples with entries in the field $%
\mathbf{F}_{q}.$ Then the mapping%
\begin{eqnarray*}
\Phi _{\vec{X}} &:&F_{n}(\mathbf{F}_{q})\rightarrow \mathbf{F}_{q}^{m} \\
f &\mapsto &\Phi _{\vec{X}}(f):=(f(\vec{x}_{1}),...,f(\vec{x}_{m}))^{t}
\end{eqnarray*}%
is a surjective linear operator. $\Phi _{\vec{X}}$ is called the \emph{%
evaluation epimorphism} \emph{of the tuple }$\vec{X}.$
\end{theorem}

\begin{proof} The proof of the linearity is left to the reader. Now let $\vec{b}\in 
\mathbf{F}_{q}^{m}$ be an arbitrary vector. Since $m\leq q^{n}$ we can
construct a function%
\begin{equation*}
g\in F_{n}(\mathbf{F}_{q})
\end{equation*}%
with the property%
\begin{equation*}
g(\vec{x}_{i})=b_{i}\text{ }\forall \text{ }i\in \{1,...,m\}
\end{equation*}%
and that means exactly%
\begin{equation*}
\Phi _{\vec{X}}(g)=\vec{b}\text{ \ \ }\endproof
\end{equation*}
\renewcommand{\endproof}{} \end{proof}

\begin{remark}[and Corollary]
\label{FullRank}Since a basis of $F_{n}(\mathbf{F}_{q})$ is given by the
fundamental monomial functions $(g_{nq\alpha })_{\alpha \in M_{q}^{n}},$ the
matrix%
\begin{equation*}
A:=(\Phi _{\vec{X}}(g_{nq\alpha }))_{\alpha \in M_{q}^{n}}\in M(m\times
q^{n};\mathbf{F}_{q})
\end{equation*}%
representing the evaluation epimorphism $\Phi _{\vec{X}}$ of the tuple $\vec{%
X}$ with respect to the basis $(g_{nq\alpha })_{\alpha \in M_{q}^{n}}$ of $%
F_{n}(\mathbf{F}_{q})$ and the canonical basis of $\mathbf{F}_{q}^{m}$ has
always the full rank $m=\min (m,q^{n}).$ That also means, that the dimension
of the $\ker (\Phi _{\vec{X}})$ is%
\begin{equation*}
\dim (\ker (\Phi _{\vec{X}}))=\dim (F_{n}(\mathbf{F}_{q}))-m=q^{n}-m
\end{equation*}
\end{remark}

\begin{corollary}
Let $\mathbf{F}_{q}$ be a finite field and $n,m\in 
%TCIMACRO{\U{2115} }%
%BeginExpansion
\mathbb{N}
%EndExpansion
$ natural numbers with $m\leq q^{n}$. Further let%
\begin{equation*}
\vec{X}:=(\vec{x}_{1},...,\vec{x}_{m})\in (\mathbf{F}_{q}^{n})^{m}
\end{equation*}%
be a tuple of $m$ different $n$-tuples with entries in the field $\mathbf{F}%
_{q}$ and $\vec{b}\in \mathbf{F}_{q}^{m}$ a vector. Then the interpolation
problem of finding a polynomial function $f\in PF_{n}(\mathbf{F}_{q})$ with
the property%
\begin{equation*}
f(\vec{x}_{i})=b_{i}\text{ }\forall \text{ }i\in \{1,...,m\}
\end{equation*}%
can be solved by solving the system of linear equations%
\begin{equation}
A\vec{y}=\vec{b}  \label{Int.Cond.}
\end{equation}%
where%
\begin{equation*}
A:=(\Phi _{\vec{X}}(g_{nq\alpha }))_{\alpha \in M_{q}^{n}}
\end{equation*}%
is the matrix representing the evaluation epimorphism $\Phi _{\vec{X}}$ of
the tuple $\vec{X}$ with respect to the basis $(g_{nq\alpha })_{\alpha \in
M_{q}^{n}}$ of $F_{n}(\mathbf{F}_{q})$ and the canonical basis of $\mathbf{F}%
_{q}^{m}$. The entries of a solution vector of the equations (\ref{Int.Cond.}%
) are the coefficients of the solution with respect to the basis $%
(g_{nq\alpha })_{\alpha \in M_{q}^{n}}.$
\end{corollary}

\begin{proof} Since $F_{n}(\mathbf{F}_{q})=PF_{n}(\mathbf{F}_{q}),$ a solution of
the interpolation problem can be found by solving the equation%
\begin{equation}
\Phi _{\vec{X}}(g)=\vec{b}  \label{OpEq}
\end{equation}%
for $g,$ where $\Phi _{\vec{X}}$\ is the surjective linear operator%
\begin{eqnarray*}
\Phi _{\vec{X}} &:&F_{n}(\mathbf{F}_{q})\rightarrow \mathbf{F}_{q}^{m} \\
f &\mapsto &\Phi _{\vec{X}}(f):=(f(\vec{x}_{1}),...,f(\vec{x}_{m}))^{t}
\end{eqnarray*}%
of the above theorem. After fixing the basis $(g_{nq\alpha })_{\alpha \in
M_{q}^{n}}$ of $F_{n}(\mathbf{F}_{q})$ and the canonical basis of $\mathbf{F}%
_{q}^{m},$ equation (\ref{OpEq}) implies the following system of linear
equations for the coefficients of the solutions with respect to the basis $%
(g_{nq\alpha })_{\alpha \in M_{q}^{n}}$%
\begin{equation*}
A\vec{y}=\vec{b}
\end{equation*}%
where%
\begin{equation*}
A:=(\Phi _{\vec{X}}(g_{nq\alpha }))_{\alpha \in M_{q}^{n}}
\end{equation*}%
is the matrix representing the map $\Phi _{\vec{X}}$ with respect to the
basis $(g_{nq\alpha })_{\alpha \in M_{q}^{n}}$ of $F_{n}(\mathbf{F}_{q})$
and the canonical basis of $\mathbf{F}_{q}^{m}$. According to Remark \ref%
{FullRank}, the matrix $A$ has full rank and therefore a solution of $A\vec{y%
}=\vec{b}$ always exists. \end{proof}

\section{Construction of special purpose symmetric bilinear forms}

Let $\mathbf{F}_{q}$ be a finite field and $n,m\in 
%TCIMACRO{\U{2115} }%
%BeginExpansion
\mathbb{N}
%EndExpansion
$ natural numbers with $m<q^{n}$. Further let%
\begin{equation*}
\vec{X}:=(\vec{x}_{1},...,\vec{x}_{m})\in \left( \mathbf{F}_{q}^{n}\right)
^{m}
\end{equation*}%
be a tuple of $m$ different $n$-tuples with entries in the field $\mathbf{F}%
_{q}$ and $d:=\dim (F_{n}(\mathbf{F}_{q})).$ Now consider the evaluation
epimorphism $\Phi _{\vec{X}}$ of the tuple $\vec{X}.$ By Remark \ref%
{FullRank} and due to the fact $m<q^{n},$ the nullity of $\Phi _{\vec{X}}$
is given by%
\begin{equation*}
s:=\dim (\ker (\Phi _{\vec{X}}))=\dim (F_{n}(\mathbf{F}_{q}))-m=q^{n}-m>0
\end{equation*}%
Now let $(u_{1},...,u_{s})$ be a basis of $\ker (\Phi _{\vec{X}})\subseteq
F_{n}(\mathbf{F}_{q}).$ By the basis extension theorem, we can extend the
basis $(u_{1},...,u_{s})$ to a basis%
\begin{equation*}
(u_{1},...,u_{s},u_{s+1},...,u_{d})
\end{equation*}%
of the whole space $F_{n}(\mathbf{F}_{q}).$ As in example \ref%
{Const.Of.Orthonorm.Basis}, we can construct a symmetric bilinear form on $%
F_{n}(\mathbf{F}_{q})$ by setting%
\begin{equation*}
\left\langle u_{i},u_{j}\right\rangle :=\delta _{ij}\text{ }\forall \text{ }%
i,j\in \{1,...,d\}
\end{equation*}%
Here the basis $\left( u_{1},...u_{d}\right) $ is orthonormal and the
vectors $(u_{s+1},...,u_{d})$ are a basis of the orthogonal complement $\ker
(\Phi _{\vec{X}})^{\perp }$ of $\ker (\Phi _{\vec{X}}).$

In general, the way we extend the basis $(u_{1},...,u_{s})$ of $\ker (\Phi _{%
\vec{X}})$ to a basis%
\begin{equation*}
(u_{1},...,u_{s},u_{s+1},...,u_{d})
\end{equation*}%
of the whole space $F_{n}(\mathbf{F}_{q})$ determines crucially the
symmetric bilinear form we get by setting $\left\langle
u_{i},u_{j}\right\rangle :=\delta _{ij}$ $\forall $ $i,j\in \{1,...,d\}.$
Consequently, the orthogonal solution of $\Phi _{\vec{X}}(g)=\vec{b}$ may
vary according to the chosen extension $u_{s+1},...,u_{d}\in F_{n}(\mathbf{F}%
_{q}).$ One systematic way to get a basis of the whole space $F_{n}(\mathbf{F%
}_{q})$ starting with a basis $(u_{1},...,u_{s})$ of $\ker (\Phi _{\vec{X}})$
is the following: let%
\begin{equation}
\left( \vec{y}_{1},...,\vec{y}_{s}\right) ^{t}  \label{MatrixU}
\end{equation}%
be the matrix whose rows are the coordinate vectors $\vec{y}_{1},...,\vec{y}%
_{s}\in K^{d}$ of $(u_{1},...,u_{s})$ with respect to the basis $%
(g_{nq\alpha })_{\alpha \in M_{q}^{n}}$ of $F_{n}(\mathbf{F}_{q}).$ Now we
perform Gauss-Jordan elimination on the matrix (\ref{MatrixU}), obtaining
the matrix $R.$ Now consider the set $B:=\{\vec{e}_{1},...,\vec{e}_{d}\}$ of
canonical unit vectors of the space $\mathbf{F}_{q}^{d}.$ For every pivot
element $r_{ij}$ used during the Gauss-Jordan elimination performed on (\ref%
{MatrixU}), eliminate the canonical unit vector $\vec{e}_{j}$ from the set $%
B.$ This yields the set $\tilde{B}.$ The coordinate vectors for a basis for
the whole space $F_{n}(\mathbf{F}_{q})$ are now given by the the rows of $R$
and the vectors in the set $\tilde{B}.$ We call this way of construction of
the orthonormal basis for the space $F_{n}(\mathbf{F}_{q})$ the \emph{%
standard orthonormalization. }We illustrate the algorithm using an example:

\begin{example}
Suppose $q=3$, $\mathbf{F}_{3}=%
%TCIMACRO{\U{2124} }%
%BeginExpansion
\mathbb{Z}
%EndExpansion
_{3},$ $m=4$, $d=3^{2}=9,$ $s=5$ and that after performing Gauss-Jordan
elimination on (\ref{MatrixU}) we get the following matrix%
\begin{equation}
R:=\left( 
\begin{array}{ccccccccc}
1 & 0 & z_{1,3} & 0 & 0 & z_{1,6} & 0 & z_{1,8} & z_{1,9} \\ 
0 & 1 & z_{2,3} & 0 & 0 & z_{2,6} & 0 & z_{2,8} & z_{2,9} \\ 
0 & 0 & 0 & 1 & 0 & z_{3,6} & 0 & z_{3,8} & z_{3,9} \\ 
0 & 0 & 0 & 0 & 1 & z_{4,6} & 0 & z_{4,8} & z_{4,9} \\ 
0 & 0 & 0 & 0 & 0 & 0 & 1 & z_{5,8} & z_{5,9}%
\end{array}%
\right)  \label{ReducedMat}
\end{equation}%
(The $z_{i,j}\in \mathbf{F}_{q}$ stand for unspecified field elements). Then
for the extension of the basis we choose the following canonical basis
vectors%
\begin{equation*}
\vec{e}_{3},\vec{e}_{6},\vec{e}_{8},\vec{e}_{9}\in 
%TCIMACRO{\U{2124} }%
%BeginExpansion
\mathbb{Z}
%EndExpansion
_{3}^{9}
\end{equation*}%
Now we substitute coordinate vectors $\left( \vec{y}_{1},...,\vec{y}%
_{5}\right) $ of the basis $(u_{1},...,u_{5})$ by the rows in the reduced
matrix \ref{ReducedMat} (this step is not strictly necessary, but it will be
needed to prove the theorems below) and get the following coordinate vectors
for a basis for the whole space $F_{2}(%
%TCIMACRO{\U{2124} }%
%BeginExpansion
\mathbb{Z}
%EndExpansion
_{3})$%
\begin{equation*}
(\widetilde{\vec{y}_{1}},...,\widetilde{\vec{y}_{s}},\vec{y}_{s+1},...,\vec{y%
}_{d}):=\left( R^{t},\vec{e}_{3},\vec{e}_{6},\vec{e}_{8},\vec{e}_{9}\right)
\end{equation*}%
In this specific example we use the standard lexicographic ordering on $%
\left( 
%TCIMACRO{\U{2115} }%
%BeginExpansion
\mathbb{N}
%EndExpansion
_{0}\right) ^{2}$ and so\ we have%
\begin{equation*}
M_{3}^{2}=\{(2,2),(2,1),(2,0),(1,2),(1,1),(1,0),(0,2),(0,1),(0,0)\}
\end{equation*}%
and%
\begin{equation*}
(g_{23\alpha }(\vec{x}))_{\alpha \in M_{3}^{2}}=\left(
x_{2}^{2}x_{1}^{2},x_{2}^{2}x_{1},x_{2}^{2},x_{2}x_{1}^{2},x_{2}x_{1},x_{2},x_{1}^{2},x_{1},1\right)
\end{equation*}%
Thus the orthonormal basis $(\widetilde{u_{1}},...,\widetilde{u_{s}}%
,u_{s+1},...,u_{d})$ of $F_{2}(%
%TCIMACRO{\U{2124} }%
%BeginExpansion
\mathbb{Z}
%EndExpansion
_{3})$ evaluated at the point $\vec{x}\in 
%TCIMACRO{\U{2124} }%
%BeginExpansion
\mathbb{Z}
%EndExpansion
_{3}^{2}$ would be%
\begin{equation*}
\left( 
\begin{array}{c}
x_{2}^{2}x_{1}^{2}+z_{1,3}x_{2}^{2}+z_{1,6}x_{2}+z_{1,8}x_{1}+z_{1,9} \\ 
x_{2}x_{1}^{2}+z_{2,3}x_{2}^{2}+z_{2,6}x_{2}+z_{2,8}x_{1}+z_{2,9} \\ 
x_{2}x_{1}^{2}+z_{3,6}x_{2}+z_{3,8}x_{1}+z_{3,9} \\ 
x_{2}x_{1}+z_{4,6}x_{2}+z_{4,8}x_{1}+z_{4,9} \\ 
x_{1}^{2}+z_{5,8}x_{1}+z_{5,9} \\ 
x_{2}^{2} \\ 
x_{2} \\ 
x_{1} \\ 
1%
\end{array}%
\right) ^{t}
\end{equation*}%
and the orthogonal solution of $\Phi _{\vec{X}}(g)=\vec{b}$ is a vector in $%
Span(x_{2}^{2}$ $,$ $x_{2}$ $,$ $x_{1}$ $,$ $1).$
\end{example}

In the next section, we will establish the exact relationship between the
orthogonal solution of $\Phi _{\vec{X}}(g)=\vec{b}$ (using the symmetric
bilinear form defined above) and the normal form with respect to the
vanishing ideal $I(X).$ This relationship can be established if the order
relation $>$ used to order the $n$-tuples in the set $M_{q}^{n}$ is a \emph{%
monomial ordering}. If, more generally, total orderings on $\left( 
%TCIMACRO{\U{2115} }%
%BeginExpansion
\mathbb{N}
%EndExpansion
_{0}\right) ^{n}$ are used to order the set $M_{q}^{n},$ the set of possible
orthogonal solutions of $\Phi _{\vec{X}}(g)=\vec{b}$ can be seen as a wider
class of normal forms (with respect to vanishing ideals) in which the
"classical" normal forms (attached to monomial orderings) appear as special
cases.

\section{Orthogonal solutions of $\Phi _{\vec{X}}(g)=\vec{b}$ and the normal
form with respect to the vanishing ideal $I(X)$}

In this section we will show the main result of this article: Given a set of
points $X\subset K^{n}$, an arbitrary polynomial $f\in K[\tau _{1},...,\tau
_{n}]$ and a monomial order $>,$ the normal form of $f$ with respect to the
vanishing ideal $I(X)\subseteq K[\tau _{1},...,\tau _{n}]$ can be calculated
as the orthogonal solution of%
\begin{equation*}
\Phi _{\vec{X}}(g)=\vec{b}
\end{equation*}%
where $\vec{b}$ is given by%
\begin{equation*}
b_{i}:=\widetilde{f}(\vec{x}_{i}),\text{ }i=1,...,m
\end{equation*}%
The yet undefined notation $\widetilde{f}$ suggests that a mapping between
the ring $K[\tau _{1},...,\tau _{n}]$ of polynomials and the vector space of
functions $F_{n}(\mathbf{F}_{q})$ is needed. That mapping will be defined
and characterized in the first lemma and theorem of this section. After
introducing some notation we arrive at an important preliminary result in
Theorem \ref{GroebnerBasis}, which states how a (particular) basis of $\ker
(\Phi _{\vec{X}})$ can be extended to a Gr\"{o}bner basis of $I(X).$ With
that result our goal can be easily reached. Please note that through this
section a more technical result stated and proved in the appendix is used.

\begin{lemma}[and Definition]
Let $K$ be a field, $n,q\in 
%TCIMACRO{\U{2115} }%
%BeginExpansion
\mathbb{N}
%EndExpansion
$ natural numbers and $K[\tau _{1},...,\tau _{n}]$ the polynomial ring in $n$
indeterminates over $K.$ Then the set of all polynomials of the form%
\begin{equation*}
\sum_{\alpha \in M_{q}^{n}}a_{\alpha }\tau _{1}^{\alpha _{1}}...\tau
_{n}^{\alpha _{n}}\in K[\tau _{1},...,\tau _{n}]
\end{equation*}%
with coefficients $a_{\alpha }\in K$ is a vector space over $K.$ We denote
this set with $P_{q}^{n}(K)\subset K[\tau _{1},...,\tau _{n}].$
\end{lemma}

\begin{proof} The easy proof is left to the reader. \end{proof}

\begin{theorem}
Let $\mathbf{F}_{q}$ be a finite field and $n\in 
%TCIMACRO{\U{2115} }%
%BeginExpansion
\mathbb{N}
%EndExpansion
$ a natural number. Then the vector spaces $P_{q}^{n}(\mathbf{F}_{q})$ and $%
F_{n}(\mathbf{F}_{q})$ are isomorphic.
\end{theorem}

\begin{proof} After defining the linear mapping%
\begin{eqnarray*}
\varphi &:&P_{q}^{n}(\mathbf{F}_{q})\rightarrow F_{n}(\mathbf{F}_{q}) \\
g &=&\sum_{\alpha \in M_{q}^{n}}a_{\alpha }\tau _{1}^{\alpha _{1}}...\tau
_{n}^{\alpha _{n}}\mapsto \varphi (g)(\vec{x}):=\sum_{\alpha \in
M_{q}^{n}}a_{\alpha }\overrightarrow{x}^{\alpha }
\end{eqnarray*}%
the claim follows easily. \end{proof}

\begin{remark}[and Definition]
The mapping $\varphi $ is defined on the set $P_{q}^{n}(K)\subset K[\tau
_{1},...,\tau _{n}],$ but of course it can naturally be extended to $K[\tau
_{1},...,\tau _{n}]$ as%
\begin{eqnarray*}
\varphi &:&K[\tau _{1},...,\tau _{n}]\rightarrow F_{n}(\mathbf{F}_{q}) \\
g &=&\sum_{\alpha \in \Gamma }a_{\alpha }\tau _{1}^{\alpha _{1}}...\tau
_{n}^{\alpha _{n}}\mapsto \varphi (g)(\vec{x}):=\sum_{\alpha \in \Gamma
}a_{\alpha }\overrightarrow{x}^{\alpha }
\end{eqnarray*}%
where $\Gamma $ is a finite set of multi indexes. We denote the image under $%
\varphi :K[\tau _{1},...,\tau _{n}]\rightarrow F_{n}(\mathbf{F}_{q})$ of a
polynomial $g\in K[\tau _{1},...,\tau _{n}]$ with%
\begin{equation*}
\widetilde{g}:=\varphi (g)\in F_{n}(\mathbf{F}_{q})
\end{equation*}
\end{remark}

\begin{definition}
Let $d\in 
%TCIMACRO{\U{2115} }%
%BeginExpansion
\mathbb{N}
%EndExpansion
$ be a natural number, $V$ a $d$-dimensional vector space over a field $K$
and $F$ a basis of $V.$ Furthermore, let $U\subset V$ be an arbitrary \emph{%
proper} subspace of $V.$ Now let $s:=\dim (U)\in 
%TCIMACRO{\U{2115} }%
%BeginExpansion
\mathbb{N}
%EndExpansion
.$ A basis $(u_{1},...,u_{s})$ of $U$ is called a \emph{cleaned kernel basis
with respect to the basis }$F$ if the matrix $\left( \vec{y}_{1},...,\vec{y}%
_{s}\right) ^{t}$ whose rows are the coordinate vectors $\vec{y}_{1},...,%
\vec{y}_{s}\in K^{d}$ of $(u_{1},...,u_{s})$ with respect to the basis $F$
is in reduced row echelon form.
\end{definition}

For a tuple $\vec{x}=(x_{1},...,x_{n})$ we write $x:=\{x_{1},...,x_{n}\}$
for the set containing all the entries in the tuple $\vec{x}.$

\begin{theorem}
\label{GroebnerBasis}Let $\mathbf{F}_{q}$ be a finite field, $n,m\in 
%TCIMACRO{\U{2115} }%
%BeginExpansion
\mathbb{N}
%EndExpansion
$ natural numbers with $m<q^{n}$ and $>$ a fixed monomial order. Further let%
\begin{equation*}
\vec{X}:=(\vec{x}_{1},...,\vec{x}_{m})\in \left( \mathbf{F}_{q}^{n}\right)
^{m}
\end{equation*}%
be a tuple of $m$ different $n$-tuples with entries in the field $\mathbf{F}%
_{q}$ and $s:=\dim (\ker (\Phi _{\vec{X}})).$ In addition, let $%
(u_{1},...,u_{s})$ be a cleaned kernel basis of $\ker (\Phi _{\vec{X}%
})\subseteq F_{n}(\mathbf{F}_{q})$ with respect to the basis $(g_{nq\alpha
})_{\alpha \in M_{q}^{n}}$. Then the family of polynomials%
\begin{equation*}
\left( \tau _{1}^{q}-\tau _{1},\tau _{2}^{q}-\tau _{2},...,\tau
_{n}^{q}-\tau _{n},\varphi ^{-1}(u_{1}),...,\varphi ^{-1}(u_{s})\right)
\end{equation*}%
is a Gr\"{o}bner basis of the vanishing ideal $I(X)\subseteq \mathbf{F}%
_{q}[\tau _{1},...,\tau _{n}]$ with respect to the monomial order $>.$
\end{theorem}

\begin{proof} \bigskip The idea of the proof is to show that%
\begin{equation*}
U:=\left( \tau _{1}^{q}-\tau _{1},\tau _{2}^{q}-\tau _{2},...,\tau
_{n}^{q}-\tau _{n},\varphi ^{-1}(u_{1}),...,\varphi ^{-1}(u_{s})\right) 
\end{equation*}%
generates the ideal $I(X)$ and that for any polynomial $g\in I(X)$ the
remainder on division of $g$ by $U$ is zero. According to a well known fact
about Gr\"{o}bner bases (see proposition 5.38 of \cite{MR1213453}) this is
equivalent to $U$ being a Gr\"{o}bner basis for $I(X).$ For this proof,
remember that the fundamental monomial functions $(g_{nq\alpha })_{\alpha
\in M_{q}^{n}}$ are ordered decreasingly with respect to the order $>.$%
\newline
Now let $g\in I(X)\subseteq \mathbf{F}_{q}[\tau _{1},...,\tau _{n}]$ be an
arbitrary polynomial in the vanishing ideal of $X.$ Since%
\begin{equation*}
\left( \tau _{1}^{q}-\tau _{1},\tau _{2}^{q}-\tau _{2},...,\tau
_{n}^{q}-\tau _{n}\right) 
\end{equation*}%
is a universal Gr\"{o}bner basis for $I(\mathbf{F}_{q}^{n})$ (see Theorem %
\ref{Th.UniversalBasis} in the appendix), there is a unique\linebreak $r\in 
\mathbf{F}_{q}[\tau _{1},...,\tau _{n}]$ with the properties

\begin{enumerate}
\item No term of $r$ is divisible by any of $LT(\tau _{1}^{q}-\tau
_{1})=\tau _{1}^{q},LT(\tau _{2}^{q}-\tau _{2})=\tau _{2}^{q},...,LT(\tau
_{n}^{q}-\tau _{n})=\tau _{n}^{q}.$ That means in particular $r\in P_{q}^{n}(%
\mathbf{F}_{q}).$

\item There is a $q\in I(\mathbf{F}_{q}^{n})$ such that $g=q+r$
\end{enumerate}

This means that when we start to divide $g$ by the (ordered) family $U$ we
get the intermediate result%
\begin{equation*}
g=q+r
\end{equation*}%
where the remainder $r\in P_{q}^{n}(\mathbf{F}_{q})$ and $q\in \left\langle
\tau _{1}^{q}-\tau _{1},\tau _{2}^{q}-\tau _{2},...,\tau _{n}^{q}-\tau
_{n}\right\rangle =I(\mathbf{F}_{q}^{n}).$ If $r=0,$ then we are done and
the remainder $\bar{g}^{U}$ on division of $g$ by $U$ is zero. If $r\neq 0,$
then we know from%
\begin{equation*}
r=g-q
\end{equation*}%
that $r\in I(X)$ ($q\in I(\mathbf{F}_{q}^{n})\subseteq I(X)$) and this is
equivalent to%
\begin{equation*}
\widetilde{r}(\vec{x})=\varphi (r)(\vec{x})=0\text{ }\forall \text{ }\vec{x}%
\in \mathbf{F}_{q}^{n}\Leftrightarrow \widetilde{r}\in \ker (\Phi _{\vec{X}})
\end{equation*}%
Since $(u_{1},...,u_{s})$ is a basis for $\ker (\Phi _{\vec{X}}),$ there are
unique $\lambda _{i}\in \mathbf{F}_{q},$ $i=1,...,s$ with%
\begin{equation*}
\widetilde{r}=\sum_{i=1}^{s}\lambda _{i}u_{i}
\end{equation*}%
Applying the vector space isomorphism $\varphi ^{-1}:F_{n}(\mathbf{F}%
_{q})\rightarrow P_{q}^{n}(\mathbf{F}_{q})$ to this equation yields%
\begin{equation*}
r=\sum_{i=1}^{s}\lambda _{i}\varphi ^{-1}(u_{i})
\end{equation*}%
From the requirement on $(u_{1},...,u_{s})$ to be a cleaned kernel basis of $%
\ker (\Phi _{\vec{X}})$ now follows for each $j\in \{1,...,s\},$ that the
leading term%
\begin{equation*}
LT(\varphi ^{-1}(u_{j}))
\end{equation*}%
doesn't appear in the polynomials $\varphi ^{-1}(u_{i}),$ $i\in
\{1,...,s\}\backslash \{j\}.$ Consequently, in the expression%
\begin{equation*}
\sum_{i=1}^{s}\lambda _{i}\varphi ^{-1}(u_{i})
\end{equation*}%
no cancellation of the leading terms $LT(\varphi ^{-1}(u_{i})),$ $i=1,...,s$
can occur. Therefore, the division of $r=\sum_{i=1}^{s}\lambda _{i}\varphi
^{-1}(u_{i})$ by $\left( \varphi ^{-1}(u_{1}),...,\varphi
^{-1}(u_{s})\right) $ must yield%
\begin{equation*}
r=\sum_{i=1}^{s}\lambda _{i}\varphi ^{-1}(u_{i})+0
\end{equation*}%
and the remainder $\bar{g}^{U}$ on division of $g$ by $U$ is zero. As a
consequence,%
\begin{equation*}
g\in \left\langle \tau _{1}^{q}-\tau _{1},\tau _{2}^{q}-\tau _{2},...,\tau
_{n}^{q}-\tau _{n},\varphi ^{-1}(u_{1}),...,\varphi ^{-1}(u_{s})\right\rangle
\end{equation*}%
and since $g\in I(X)$ was arbitrary%
\begin{equation*}
I(X)\subseteq \left\langle \tau _{1}^{q}-\tau _{1},\tau _{2}^{q}-\tau
_{2},...,\tau _{n}^{q}-\tau _{n},\varphi ^{-1}(u_{1}),...,\varphi
^{-1}(u_{s})\right\rangle
\end{equation*}%
The inclusion%
\begin{equation*}
\left\langle \tau _{1}^{q}-\tau _{1},\tau _{2}^{q}-\tau _{2},...,\tau
_{n}^{q}-\tau _{n},\varphi ^{-1}(u_{1}),...,\varphi
^{-1}(u_{s})\right\rangle \subseteq I(X)
\end{equation*}%
is given by the fact $u_{1},...,u_{s}\in \ker (\Phi _{\vec{X}})$ and Theorem %
\ref{Th.UniversalBasis}. Summarizing we can say%
\begin{equation*}
\left\langle \tau _{1}^{q}-\tau _{1},\tau _{2}^{q}-\tau _{2},...,\tau
_{n}^{q}-\tau _{n},\varphi ^{-1}(u_{1}),...,\varphi
^{-1}(u_{s})\right\rangle =I(X)
\end{equation*}%
and for every $g\in I(X)$ the remainder $\bar{g}^{U}$ on division of $g$ by $%
U$ is zero. Now proposition 5.38 of \cite{MR1213453} (see also the remarks
after corollary 2, chapter 2, \S\ 6 of \cite{MR1417938}) proves the claim. 
\end{proof}

\begin{theorem}
Let $\mathbf{F}_{q}$ be a finite field, $n,m\in 
%TCIMACRO{\U{2115} }%
%BeginExpansion
\mathbb{N}
%EndExpansion
$ natural numbers with $m<q^{n}$ and $>$ a fixed monomial order. Further let%
\begin{equation*}
\vec{X}:=(\vec{x}_{1},...,\vec{x}_{m})\in (\mathbf{F}_{q}^{n})^{m}
\end{equation*}%
be a tuple of $m$ different $n$-tuples with entries in the field $\mathbf{F}%
_{q}$, $\vec{b}\in \mathbf{F}_{q}^{m}$ a vector, $d:=\dim (F_{n}(\mathbf{F}%
_{q}))$ and $s:=\dim (\ker (\Phi _{\vec{X}})).$ In addition, let $%
(u_{1},...,u_{s})$ be a \emph{cleaned} kernel basis of $\ker (\Phi _{\vec{X}%
})\subseteq F_{n}(\mathbf{F}_{q})$ with respect to the basis $(g_{nq\alpha
})_{\alpha \in M_{q}^{n}}$, $(u_{1},...,u_{s},u_{s+1},...,u_{d})$ an
orthonormal basis of $F_{n}(\mathbf{F}_{q})$ constructed \emph{using the
standard orthonormalization} and $f\in \mathbf{F}_{q}[\tau _{1},...,\tau
_{n}]$ a polynomial satisfying the interpolation conditions%
\begin{equation*}
\widetilde{f}(\vec{x}_{j})=b_{j}\text{ }\forall \text{ }j\in \{1,...,m\}
\end{equation*}%
Furthermore, let $U\subseteq I(X)$ be an arbitrary Gr\"{o}bner basis of the
vanishing ideal $I(X)$ with respect to the monomial order $>$ and $v^{\ast }$
the orthogonal solution of $\Phi _{\vec{X}}(g)=\vec{b}$. Then%
\begin{equation*}
\varphi ^{-1}(v^{\ast })=\overline{f}^{U}
\end{equation*}
\end{theorem}

\begin{proof} If $\varphi ^{-1}(v^{\ast })=0$ then $v^{\ast }=0$ and%
\begin{equation*}
\vec{b}=\Phi _{\vec{X}}(v^{\ast })=\Phi _{\vec{X}}(0)=\vec{0}
\end{equation*}%
In this case we also have%
\begin{equation*}
\overline{f}^{U}=0
\end{equation*}%
and therefore%
\begin{equation*}
\varphi ^{-1}(v^{\ast })=\overline{f}^{U}
\end{equation*}%
Assume $\varphi ^{-1}(v^{\ast })\neq 0.$ Since the remainder on division by
a Gr\"{o}bner basis is independent of which Gr\"{o}bner basis we use (for a
fixed monomial order), the idea of the proof is to show that $\varphi
^{-1}(v^{\ast })$ is the unique remainder on division by the Gr\"{o}bner
basis%
\begin{equation*}
\left( \tau _{1}^{q}-\tau _{1},\tau _{2}^{q}-\tau _{2},...,\tau
_{n}^{q}-\tau _{n},\varphi ^{-1}(u_{1}),...,\varphi ^{-1}(u_{s})\right)
\end{equation*}%
(see Theorem \ref{GroebnerBasis}). Now, since $\varphi ^{-1}(v^{\ast })\in
P_{q}^{n}(\mathbf{F}_{q}),$ no term of $\varphi ^{-1}(v^{\ast })$ is
divisible by any of the%
\begin{equation*}
LT(\tau _{1}^{q}-\tau _{1})=\tau _{1}^{q},LT(\tau _{2}^{q}-\tau _{2})=\tau
_{2}^{q},...,LT(\tau _{n}^{q}-\tau _{n})=\tau _{n}^{q}
\end{equation*}%
If terms of $\varphi ^{-1}(v^{\ast })$ would be divisible by%
\begin{equation*}
LT(\varphi ^{-1}(u_{1})),...,LT(\varphi ^{-1}(u_{s}))
\end{equation*}%
then after division by the family%
\begin{equation*}
\left( \tau _{1}^{q}-\tau _{1},\tau _{2}^{q}-\tau _{2},...,\tau
_{n}^{q}-\tau _{n},\varphi ^{-1}(u_{1}),...,\varphi ^{-1}(u_{s})\right)
\end{equation*}%
we would have%
\begin{equation}
\varphi ^{-1}(v^{\ast })=\sum_{i=1}^{s}h_{i}\varphi ^{-1}(u_{i})+r
\label{Relation}
\end{equation}%
where $h_{i},r\in \mathbf{F}_{q}[\tau _{1},...,\tau _{n}],$ $i=1,...,s$ and
either $r=0$ or no term of $r$ is divisible by the%
\begin{equation*}
LT(\tau _{1}^{q}-\tau _{1}),...,LT(\tau _{n}^{q}-\tau _{n}),LT(\varphi
^{-1}(u_{1})),...,LT(\varphi ^{-1}(u_{s}))
\end{equation*}%
If $r=0$, then%
\begin{equation*}
\varphi ^{-1}(v^{\ast })=\sum_{i=1}^{s}h_{i}\varphi ^{-1}(u_{i})
\end{equation*}%
and the polynomial $\varphi ^{-1}(v^{\ast })$ vanishes on the set $X,$ that
is%
\begin{equation*}
\varphi (\varphi ^{-1}(v^{\ast }))(\vec{x})=v^{\ast }(\vec{x})=0\text{ }%
\forall \text{ }\vec{x}\in X
\end{equation*}%
Consequently%
\begin{equation*}
\vec{b}=\Phi _{\vec{X}}(v^{\ast })=\vec{0}
\end{equation*}%
and due to the uniqueness of the orthogonal solution%
\begin{equation*}
v^{\ast }=0
\end{equation*}%
But this is a contradiction to our assumption $\varphi ^{-1}(v^{\ast })\neq
0.$\newline
Now if $r\neq 0,$ since no term of $r$ is divisible by $LT(\tau
_{1}^{q}-\tau _{1}),...,LT(\tau _{n}^{q}-\tau _{n}),$ then in particular $%
r\in P_{q}^{n}(\mathbf{F}_{q}).$ Due to the fact, that $%
(u_{1},...,u_{s},u_{s+1},...,u_{d})$ is a basis for $F_{n}(\mathbf{F}_{q}),$
we can write%
\begin{equation*}
\widetilde{r}=\varphi (r)=\sum_{j=1}^{d}\lambda _{j}u_{j}
\end{equation*}%
with unique $\lambda _{j}\in \mathbf{F}_{q},$ $j=1,...,d.$ Applying the
vector space isomorphism $\varphi ^{-1}:F_{n}(\mathbf{F}_{q})\rightarrow
P_{q}^{n}(\mathbf{F}_{q})$ to this equation yields%
\begin{equation*}
r=\sum_{j=1}^{d}\lambda _{j}\varphi ^{-1}(u_{j})
\end{equation*}%
From the requirement on $(u_{1},...,u_{s})$ to be a cleaned kernel basis of $%
\ker (\Phi _{\vec{X}})$ with respect to the basis $(g_{nq\alpha })_{\alpha
\in M_{q}^{n}}$ and since the basis extension $%
(u_{1},...,u_{s},u_{s+1},...,u_{d})$ has been constructed using the standard
orthonormalization, in the expression%
\begin{equation*}
\sum_{j=1}^{d}\lambda _{j}\varphi ^{-1}(u_{j})
\end{equation*}%
no cancellation of the leading terms $LT(\varphi ^{-1}(u_{k})),$ $k=1,...,s$
can occur. But $r$ is not divisible by $LT(\varphi
^{-1}(u_{1})),...,LT(\varphi ^{-1}(u_{s}))$ and that forces%
\begin{equation*}
\lambda _{k}=0,\text{ }\forall \text{ }k\in \{1,...,s\}
\end{equation*}%
In other words%
\begin{equation*}
r=\sum_{j=s+1}^{d}\lambda _{j}\varphi ^{-1}(u_{j})\Leftrightarrow \widetilde{%
r}=\varphi (r)=\sum_{j=s+1}^{d}\lambda _{j}u_{j}
\end{equation*}%
which is equivalent to%
\begin{equation}
\widetilde{r}\in \ker (\Phi _{\vec{X}})^{\perp }  \label{Ortho}
\end{equation}%
From the equation (\ref{Relation}) we know that%
\begin{equation*}
r=\varphi ^{-1}(v^{\ast })-\sum_{i=1}^{s}h_{i}\varphi ^{-1}(u_{i})
\end{equation*}%
and that means%
\begin{equation*}
\widetilde{r}(\vec{x})=v^{\ast }(\vec{x})\text{ }\forall \text{ }\vec{x}\in X
\end{equation*}%
In other words%
\begin{equation*}
\Phi _{\vec{X}}(\widetilde{r})=\vec{b}
\end{equation*}%
This together with (\ref{Ortho}) says that $\widetilde{r}$ is an orthogonal
solution of $\Phi _{\vec{X}}(g)=\vec{b}.$ From the uniqueness now follows%
\begin{equation*}
v^{\ast }=\widetilde{r}\Leftrightarrow \varphi ^{-1}(v^{\ast })=r
\end{equation*}%
Consequently, no term of the polynomial $\varphi ^{-1}(v^{\ast })$ is
divisible by any of the leading terms of the elements of the Gr\"{o}bner
basis (see Theorem \ref{GroebnerBasis})%
\begin{equation*}
G:=\left( \tau _{1}^{q}-\tau _{1},\tau _{2}^{q}-\tau _{2},...,\tau
_{n}^{q}-\tau _{n},\varphi ^{-1}(u_{1}),...,\varphi ^{-1}(u_{s})\right)
\end{equation*}%
for the vanishing ideal $I(X).$ Now we define the polynomial%
\begin{equation*}
h:=f-\varphi ^{-1}(v^{\ast })
\end{equation*}%
Since $v^{\ast }$ is a solution of $\Phi _{\vec{X}}(g)=\vec{b}$ and $f$
satisfies the interpolation conditions%
\begin{equation*}
\widetilde{f}(\vec{x}_{j})=b_{j}\text{ }\forall \text{ }j\in \{1,...,m\}
\end{equation*}%
we have%
\begin{equation*}
\widetilde{h}(\vec{x})=\widetilde{f}(\vec{x})-v^{\ast }(\vec{x})=0\text{ }%
\forall \text{ }\vec{x}\in X\Leftrightarrow h\in I(X)
\end{equation*}%
So we have a polynomial $h\in I(X)$ such that%
\begin{equation*}
f=h+\varphi ^{-1}(v^{\ast })
\end{equation*}%
By proposition 1, chapter 2, \S 6 in \cite{MR1417938}, $\varphi
^{-1}(v^{\ast })$ is the unique remainder on division by the Gr\"{o}bner
basis $G.$ It is a well known fact, that the remainder on division by a Gr%
\"{o}bner basis is independent of which Gr\"{o}bner basis we use, as long as
we use one fixed particular monomial order. Therefore%
\begin{equation*}
\overline{f}^{U}=\overline{f}^{G}=\varphi ^{-1}(v^{\ast })\text{ \ \ }%
\endproof
\end{equation*}%
\renewcommand{\endproof}{} \end{proof}

\begin{remark}[and main theorem]
Let $\mathbf{F}_{q}$ be a finite field, $n,m\in 
%TCIMACRO{\U{2115} }%
%BeginExpansion
\mathbb{N}
%EndExpansion
$ natural numbers with $m<q^{n}$ and $>$ a fixed monomial order. Further let%
\begin{equation*}
\vec{X}:=(\vec{x}_{1},...,\vec{x}_{m})\in (\mathbf{F}_{q}^{n})^{m}
\end{equation*}%
be a tuple of $m$ different $n$-tuples with entries in the field $\mathbf{F}%
_{q}$, $U\subseteq I(X)$ an arbitrary Gr\"{o}bner basis of the vanishing
ideal $I(X)$ and $f\in \mathbf{F}_{q}[\tau _{1},...,\tau _{n}]$ an \emph{%
arbitrary} polynomial. Then%
\begin{equation*}
\overline{f}^{U}=\varphi ^{-1}(v^{\ast })
\end{equation*}%
where $v^{\ast }$ is the orthogonal solution of $\Phi _{\vec{X}}(g)=\vec{b}$
and $\vec{b}$ is given by%
\begin{equation*}
b_{i}:=\widetilde{f}(\vec{x}_{i}),\text{ }i=1,...,m
\end{equation*}
\end{remark}

\begin{remark}
\label{MatrixCalculation}Let%
\begin{equation*}
A:=(\Phi _{\vec{X}}(g_{nq\alpha }))_{\alpha \in M_{q}^{n}}\in M(m\times
q^{n};\mathbf{F}_{q})
\end{equation*}%
be the matrix representing the evaluation epimorphism $\Phi _{\vec{X}}$ of
the tuple $\vec{X}$ with respect to the basis $(g_{nq\alpha })_{\alpha \in
M_{q}^{n}}$ of $F_{n}(\mathbf{F}_{q})$ and the canonical basis of $\mathbf{F}%
_{q}^{m}$ and $S$ the matrix%
\begin{equation*}
S_{ij}:=\left\langle g_{nq\alpha _{i}},g_{nq\alpha _{j}}\right\rangle ,\text{
}i,j\in \{1,...,q^{n}\}
\end{equation*}%
representing the symmetric bilinear form\ with respect to the basis $%
(g_{nq\alpha })_{\alpha \in M_{q}^{n}}$. Further let\linebreak $\vec{y}%
_{1},...,\vec{y}_{s}\in \mathbf{F}_{q}^{d}$ be the coordinate vectors of $%
(u_{1},...,u_{s})$ with respect to the basis $(g_{nq\alpha })_{\alpha \in
M_{q}^{n}}$. Then the above result states that the normal form $\overline{f}%
^{U}$ of $f$ with respect to the Gr\"{o}bner basis $U\subseteq I(X)$ can be
calculated by solving the following system of inhomogeneous linear equations%
\begin{eqnarray*}
A\vec{z} &=&\vec{b} \\
\vec{y}_{i}^{t}S\vec{z} &=&0,\text{ }i=1,...,s
\end{eqnarray*}
\end{remark}

\section{Acknowledgements}

We would like to thank Dr. Gretchen Matthews, Dr. Michael Shapiro and Dr.
Michael Stillman for very helpful comments and contributions for the content
of this paper.

\section{Appendix}

\begin{lemma}
Let $K$ be a field, $n\in 
%TCIMACRO{\U{2115} }%
%BeginExpansion
\mathbb{N}
%EndExpansion
$ a natural number, $K[\tau _{1},...,\tau _{n}]$ the polynomial ring in $n$
indeterminates over $K$ and $>$ an arbitrary monomial order. Then for each
natural number $m\in 
%TCIMACRO{\U{2115} }%
%BeginExpansion
\mathbb{N}
%EndExpansion
$ and each $i\in \{1,...,n\}$ it holds%
\begin{equation}
\tau _{i}^{m}>\tau _{i}^{m-1}>...>\tau _{i}>\tau _{i}^{0}  \label{Inequa.}
\end{equation}
\end{lemma}

\begin{proof}The claim follows from the well-ordering, the translation invariance
and transitivity of $>.$\end{proof}

\begin{theorem}
Let $\mathbf{F}_{q}$ be a finite field and $n\in 
%TCIMACRO{\U{2115} }%
%BeginExpansion
\mathbb{N}
%EndExpansion
$ a natural number. Then the family of polynomials%
\begin{equation*}
\left( \tau _{1}^{q}-\tau _{1},\tau _{2}^{q}-\tau _{2},...,\tau
_{n}^{q}-\tau _{n}\right)
\end{equation*}%
is a basis for the vanishing ideal%
\begin{equation*}
I(\mathbf{F}_{q}^{n})\subseteq \mathbf{F}_{q}[\tau _{1},...\tau _{n}]
\end{equation*}
\end{theorem}

\begin{proof}The proof of this well-known result can be found after Lemma 3.1 of 
\cite{Germundsson}.\end{proof}

\begin{theorem}
\label{Th.UniversalBasis}Let $\mathbf{F}_{q}$ be a finite field and $n\in 
%TCIMACRO{\U{2115} }%
%BeginExpansion
\mathbb{N}
%EndExpansion
$ a natural number. Then the family of polynomials%
\begin{equation*}
\left( \tau _{1}^{q}-\tau _{1},\tau _{2}^{q}-\tau _{2},...,\tau
_{n}^{q}-\tau _{n}\right)
\end{equation*}%
is a universal Gr\"{o}bner basis for the vanishing ideal%
\begin{equation*}
I(\mathbf{F}_{q}^{n})\subseteq \mathbf{F}_{q}[\tau _{1},...\tau _{n}]
\end{equation*}
\end{theorem}

\begin{proof} From the inequalities \ref{Inequa.} it follows in particular for all
possible monomial orders%
\begin{equation*}
LM(\tau _{i}^{q}-\tau _{i})=\tau _{i}^{q}\text{ }\forall \text{ }i\in
\{1,...,n\}
\end{equation*}%
As a consequence, for the least common multiple ($LCM$) of $LM(\tau
_{j}^{q}-\tau _{j})$ and $LM(\tau _{i}^{q}-\tau _{i}),$ $i\neq j$ holds%
\begin{equation*}
LCM(LM(\tau _{j}^{q}-\tau _{j}),LM(\tau _{i}^{q}-\tau _{i}))=LCM(\tau
_{j}^{q},\tau _{i}^{q})=\tau _{j}^{q}\tau _{i}^{q}\text{ }\forall \text{ }%
i,j\in \{1,...,n\}\text{ with }i\neq j
\end{equation*}%
and for the $S$-polynomial of $\tau _{j}^{q}-\tau _{j}$ and $\tau
_{i}^{q}-\tau _{i},$ $i\neq j$ we have%
\begin{equation*}
S(\tau _{j}^{q}-\tau _{j},\tau _{i}^{q}-\tau _{i})=\tau _{i}^{q}(\tau
_{j}^{q}-\tau _{j})-\tau _{j}^{q}(\tau _{i}^{q}-\tau _{i})=\tau _{j}^{q}\tau
_{i}-\tau _{i}^{q}\tau _{j}\text{ }\forall \text{ }i,j\in \{1,...,n\}\text{
with }i\neq j
\end{equation*}%
Now let's divide $S(\tau _{j}^{q}-\tau _{j},\tau _{i}^{q}-\tau _{i})=\tau
_{j}^{q}\tau _{i}-\tau _{i}^{q}\tau _{j}$ by $\left( \tau _{1}^{q}-\tau
_{1},\tau _{2}^{q}-\tau _{2},...,\tau _{n}^{q}-\tau _{n}\right) .$ Without
loss of generality let%
\begin{equation*}
\tau _{j}^{q}\tau _{i}>\tau _{i}^{q}\tau _{j}
\end{equation*}%
(which is equivalent to $LT(\tau _{j}^{q}\tau _{i}-\tau _{i}^{q}\tau
_{j})=\tau _{j}^{q}\tau _{i}$). Then, after the first division step, we get
the remainder%
\begin{equation*}
-\tau _{i}^{q}\tau _{j}+\tau _{i}\tau _{j}
\end{equation*}%
Now we know from the inequalities (\ref{Inequa.}) after translation by $\tau
_{j}$%
\begin{equation*}
\tau _{i}^{q}\tau _{j}>\tau _{i}\tau _{j}\Rightarrow LT(-\tau _{i}^{q}\tau
_{j}+\tau _{i}\tau _{j})=-\tau _{i}^{q}\tau _{j}
\end{equation*}%
so we can continue the division process and we get the remainder%
\begin{equation*}
-\tau _{i}^{q}\tau _{j}+\tau _{i}\tau _{j}-(-\tau _{j})(\tau _{i}^{q}-\tau
_{i})=0
\end{equation*}%
By the theorem above%
\begin{equation*}
I(\mathbf{F}_{q}^{n})=\left\langle \tau _{1}^{q}-\tau _{1},\tau
_{2}^{q}-\tau _{2},...,\tau _{n}^{q}-\tau _{n}\right\rangle
\end{equation*}%
And so, according to Buchberger's $S$-pair criterion (see Theorem 6 of
chapter 2, \S 6 in \cite{MR1417938}),%
\begin{equation*}
\left( \tau _{1}^{q}-\tau _{1},\tau _{2}^{q}-\tau _{2},...,\tau
_{n}^{q}-\tau _{n}\right)
\end{equation*}%
is a universal Gr\"{o}bner Basis for $I(\mathbf{F}_{q}^{n}).$ \end{proof}

%\setstretch{1}

\bibliographystyle{authordate1}
\bibliography{MathRef}

\end{document}